\documentclass[12pt, reqno]{amsart}
\usepackage{amsmath,amssymb,amsthm,graphicx,mathrsfs,url} 
\usepackage{amsmath,amssymb,amsthm,tikz}
\usetikzlibrary{decorations.pathreplacing}
\usepackage[toc,page]{appendix}
\usepackage{comment}
\usepackage[dvipsnames]{xcolor}
\usepackage[margin=0.9in]{geometry}
\usepackage[colorlinks=true, citecolor=Green]{hyperref}
\usepackage[english]{babel}
\usepackage[T1]{fontenc}
\usetikzlibrary{arrows.meta,calc}
\usepackage{amsthm}
\usepackage{ragged2e} 
\usepackage{etoolbox}
\AtBeginEnvironment{thebibliography}{\small}

\usetikzlibrary{decorations.pathreplacing}

\newcommand{\eps}{\varepsilon}

\newtheorem{theorem}{Theorem}[section]
\newtheorem{lemma}[theorem]{Lemma}

\newtheorem{conjecture}[theorem]{Conjecture}
\newtheorem{corollary}[theorem]{Corollary}

\newtheorem{proposition}[theorem]{Proposition}
\newtheorem*{theoremfrench*}{Théorème}
\newtheorem*{theorem*}{Theorem}

\numberwithin{equation}{section}
\author{Tristan Humbert}
\email{humbertt@imj-prg.fr}
\address{Sorbonne Université, Paris France 75005.}

\begin{document}
\begin{abstract}
Let $(M,g)$ be a closed negatively curved surface. If $g$ is strictly $\tfrac 19$-pinched and has the same unmarked length spectrum as a hyperbolic metric, we show that $g$ is hyperbolic. As a consequence, we show that any hyperbolic metric on a surface admits a $C^2$-neighborhood in the space of metrics in which it is characterized by its unmarked length spectrum, up to isometry.
\end{abstract}
\title{Local unmarked length spectrum rigidity for hyperbolic surfaces} 
\maketitle
\section{Introduction}
\subsection{Setting}
Let $(M^n,g)$ be a smooth closed $n$-dimensional manifold equipped with a smooth Riemannian metric $g$ of negative sectional curvature.  Let $\mathcal C$ denote the set of free homotopy classes of closed curves on $M$. We define the \emph{marked length spectrum} of the metric $g$ to be
$$\mathcal L_g: \mathcal C\to (0,+\infty), \quad c\mapsto \ell_g(\gamma_g(c)), $$
where $\ell_g$ is the length for the metric $g$ of $\gamma_g(c)$, the unique $g$-geodesic in the class $c$.

Another important invariant is given by the \emph{unmarked length spectrum}, namely, the collection of lengths of closed geodesics:
\begin{equation}
\label{LS}
L(g):=\{\ell_g(\gamma)\mid \gamma \text{ is a closed } g\text{-geodesic} \}\subset (0,\infty),
\end{equation}
where the lengths of closed geodesics are counted with multiplicity.

It is a long-standing conjecture of Burns and Katok  \cite{burnskatok} that two negatively curved metrics are isometric if and only if their marked length spectra are equal. Katok proved it \cite{Ka} when the two metrics are in the same conformal class. Then Croke and Otal independently solved the conjecture for surfaces \cite{Cro90,Ota90}. Their result was extended recently to surfaces with Anosov geodesic flow by Guillarmou, Lefeuvre and Paternain \cite{GuLefPa}.

The conjecture is still open in higher dimension. However, Hamenstädt \cite{Ham99} solved the conjecture if one of the metrics is locally symmetric. More recently, Guillarmou and Lefeuvre \cite{GL19} (see also the related work of Guillarmou, Knieper and Lefeuvre \cite{GuKnLef}) showed the conjecture locally, that is, if the two metrics are close enough in some $C^k$-topology. We note that Butt obtained quantitative versions of marked length spectrum rigidity on surfaces and of Hamenstädt's result \cite{But1,Butt26,But25}.

It is known since the work of Vigneras \cite{Vig80} (see also \cite{Sun85}), that there exists non-isometric hyperbolic metrics with the same length spectra. Hence, the length spectrum cannot characterize  negatively curved metrics up to isometry.

Nevertheless, in view of the recent breakthrough of Guillarmou and Lefeuvre on the local rigidity of the marked length spectrum, it seems reasonable to state the following conjecture.
\begin{conjecture}
\label{1}
Let $(M^n,g)$ be a negatively curved closed manifold of dimension $n\geq 2$. There exists $k\in \mathbb N$ and $\epsilon>0$ such that for any negatively curved metric $g'$ with $\|g-g'\|_{C^k}<\epsilon$, one has $L(g)=L(g')$ if and only if $g$ and $g'$ are isometric.
\end{conjecture}
The conjecture seems to be folklore among specialists of the marked length spectrum. The length spectrum is known to be rigid along deformations by the works of Guillemin, Kazhdan \cite{GK80} and Croke, Sharafutdinov \cite{CS98}. In a recent preprint \cite{DWWMRO}, DeWitt, Durham, Marshall Reber and O'Hare proved the conjecture when the metric $g$ has \emph{exponentially separated length spectrum}, see \cite[Equation (1.1.1)]{DWWMRO}. This condition is known to hold for a $C^k$-dense set of negatively curved metrics (but is not generic) \cite{DJ16,Schenck2020}.
\subsection{Statements of results}
The purpose of this short note is to prove Conjecture \ref{1} when $g_0$ is a hyperbolic metric on a surface and $k=2$.
\begin{theorem}
\label{theo:main}
Let $(M^2,g_0)$ be a closed hyperbolic surface. Then there exists $\epsilon>0$ such that for any negatively curved metric with $\|g-g_0\|_{C^2}<\epsilon$, we have $L(g)=L(g_0)$ if and only if $g$ is isometric to $g_0$.
\end{theorem}
Note that Theorem \ref{theo:main} was shown in \cite{Gogolev18} under the additional hypothesis that $\mathrm{Area}_g(M)=\mathrm{Area}_{g_0}(M)$. 
We would like to stress that our result does not follow directly from  \cite[Theorem A]{DWWMRO}. Indeed, the exponentially separated length spectrum condition is known to fail for a Baire generic set of hyperbolic metrics on a surface, see \cite[Theorem 3.1]{DJ16}. Moreover, we obtain the rigidity statement in a $C^2$-neighborhood of $g_0$ while  \cite[Theorem A]{DWWMRO} is obtained for a $C^k$-neighborhood where $k$ is typically larger than $2$ (see the proof of \cite[Theorem 5.2.1]{DWWMRO}).

Note however that hyperbolic metrics in dimension at least three always have exponentially separated length spectrum \cite[Corollary 2.7]{DJ16}. Hence, as a consequence of Theorem \ref{theo:main} and \cite[Corollary A]{DWWMRO}, we deduce.
\begin{corollary}
Let $(M^n,g_0)$ be a closed hyperbolic manifold of dimension $n\geq 2$, then there exists a $C^\infty$-neighborhood $\mathcal U$ of $g_0$ in the space of negatively curved metrics such that for any $g\in \mathcal U$, one has $L(g)=L(g_0)$ if and only if $g$ and $g_0$ are isometric.
\end{corollary}
This should be compared with the local audibility result on negatively curved locally symmetric manifolds of Sharafutdinov \cite{Sha09}, where it is shown that locally symmetric metrics admit a $C^{\infty}$-neighborhood in which they are determined, up to isometry, by the spectrum of their Laplace operator. 

On hyperbolic surfaces, the Selberg trace formula implies that the Laplace  and length spectra (counted with multiplicities) determine each other \cite{Hu}. This equivalence does not persist for variable curvature metrics: the wave trace formula \cite{DG75} determines the set of lengths of closed geodesics, but in general not their multiplicities. Thus, although the two results are closely related in spirit, neither is a direct consequence of the other.

Theorem \ref{theo:main} is  actually a consequence of the following result.We denote by $K_g$ the  sectional curvature of a metric $g$. Recall that for $\alpha\in (0,1]$, we say that a metric $g$ is strictly $\alpha$-pinched~if \begin{equation}
\label{eq:pinching}
\exists C>0, \quad -C<K_g<-\alpha C.
\end{equation}

\begin{theorem}
\label{theo:hyp}
Let $(M^2,g_0)$ be a closed hyperbolic surface. Let $g$ be a strictly $\tfrac 19$-pinched negatively curved metric on $M$ and suppose that $L(g)=L(g_0)$. Then $g$ is a hyperbolic metric.
\end{theorem}

As explained above, Vigneras constructed examples of non-isometric hyperbolic metrics with the same length spectra. We prove however in Lemma \ref{lemma:final} that any hyperbolic metric admits a $C^0$-neighborhood in the space of hyperbolic metrics in which it is characterized by its length spectrum. Theorem \ref{theo:main} then follows from this observation and Theorem \ref{theo:hyp}. 

Let us record another corollary of Theorem \ref{theo:hyp} and the work of Reid \cite{Reid1992,Reid2014}. Recall that two closed hyperbolic surfaces are said to be \emph{commensurable} if they have a common finite-sheeted cover. As a direct consequence of Theorem \ref{theo:hyp} and \cite[Theorem 5.1]{Reid2014}, we obtain:
\begin{corollary}
\label{cor:Reid}
Let $(M^2,g_0)$ be a closed arithmetic hyperbolic surface. Let $g$ be a strictly $\tfrac 19$-pinched negatively curved metric on $M$ and suppose that $L(g)=L(g_0)$. Then $g$ and $g_0$ are commensurable.
\end{corollary}

\textbf{Acknowledgements.} The author would like to thank Colin Guillarmou and Thibault Lefeuvre for their guidance and advice during the writing of this paper. The author would also like to thank Frédéric Faure for encouraging discussions and for answering some questions on the paper \cite{FT}. Finally, the author would like to thank Andrey Gogolev for pointing out that Corollary \ref{cor:Reid} follows from our main theorem and giving us more details on the history of Conjecture \ref{1}.

This research was supported by the European Research Council (ERC) under
the European Union’s Horizon 2020 research and innovation programme (Grant agreement no. 101162990 — ADG).
\section{Preliminaries}
\subsection{Anosov flows}Let $(M^2,g)$ be a closed negatively curved surface. Let $(\varphi_t^g)_{t \in \mathbb R}$ denote the geodesic flow on the unit tangent bundle $S^gM:=\{(x,v)\in TM , \|v\|_g=1\}$ of $M$.
The geodesic flow is an \emph{Anosov flow}, which means that if we denote by $X^g=\tfrac{d}{dt}\varphi_t^g|_{t=0}$ the geodesic vector field, then there exists a flow-invariant, continuous splitting 
\begin{equation}
\label{eq:Anosov}
T(S^gM)=E_u^g\oplus \mathbb R X^g\oplus E_s^g,
\end{equation}of the tangent bundle such that the differential of the flow $d\varphi_t^g$ contracts (resp. expands) exponentially on $E_s^g$ (resp. $E_u^g$).
\subsection{Entropies}
\label{sec:entropy}
Denote by $\mathcal C$ the set of free homotopy classes of closed loops on $M$ and recall that any non-trivial free homotopy class $c$ contains a unique $g$-geodesic $\gamma_g(c)$, see for instance \cite[Theorem 13.12]{GuMaz}. The topological entropy is given by (see \cite[Theorem 7.6.9]{FishHas})
\begin{equation}
\label{eq:top2}
h_{\mathrm{top}}(\varphi_1)=\lim_{T\to+\infty}\frac 1T \ln \mathrm{Card}\{c \in \mathcal C \mid  \ell_g(\gamma_g(c))\leq T\}.
\end{equation}
In other words, the topological entropy is the exponential growth rate, when $T\to+\infty$, of the number of closed geodesics of length less than $T$. 
We define the \emph{stable Jacobian}:
\begin{equation}
\label{eq:Js1}
J^s_g(x):=\frac{d}{dt}\ln\mathrm{det}(d\varphi_t^g(x)_{|E_s(x)})|_{t=0}.
\end{equation}
 Let $\mu_g$ denote the \emph{Liouville measure} associated to the metric $g$. Recall that it is the volume form associated to the Sasaki metric on $SM$, normalized so that $\mu_g(SM)=1$, see for instance \cite[Chapter 1.7]{GuMaz}. We let $h_{\mathrm{Liou}}(g)$ denote the \emph{Liouville entropy}, i.e., the metric entropy of the geodesic flow $(\varphi^g_t)_{t\in \mathbb R}$ with respect to $\mu_g$. 
 We have \begin{equation}
 \label{eq:SRBentropy}
 h_{\mathrm{Liou}}(g)=-\int_{S^gM}J^s_g(v)d\mu_g(v),
 \end{equation}
 see for instance \cite[Corollary 7.4.5]{FishHas}.
\subsection{Pollicott-Ruelle resonances} 
 For $0\leq k \leq 2$, let
$$ \mathscr E_0^k:=\{ u\in C^{\infty}(S M; \Lambda^k T^* (SM)) \mid \iota_{X}u=0\}, $$
where $\iota_X$ denotes the contraction by $X$. The vector field $X$ acts by Lie derivative:
\begin{equation}
\label{eq:Lie}
\mathbf{X}_k^g:\mathscr E_0^k\to \mathscr E_0^k, \quad  \mathbf{X}_k^g \omega:= \mathcal L_{X^g} \omega= \frac{d}{dt}|_{t=0}(\varphi_t^g)^*\omega.
\end{equation}
In this setting, it is now well understood that one can associate to $\mathscr E_0^k$ a discrete spectrum $\sigma(\mathbf X_k^g)\subset \mathbb C$, the \emph{Pollicott-Ruelle resonances}, by making $\mathbf {X}_k^g$ act on specially designed \emph{anisotropic spaces}, see for instance \cite{BKL,BL,BT,GouLiv,Fau08,Fau10,FT}. The set $\sigma(\mathbf X_k^g)$ is intrinsic to the Anosov flow and governs important dynamical properties such as the decay rate of correlations \cite{Liv,TsuZh}.

For our purpose, it will be sufficient to know that any  resonance $\lambda_0\in \sigma(\mathbf X_k^g)$ has finite algebraic multiplicity, which we will denote by $m_k^g(\lambda_0)$.
 \subsection{Ruelle zeta function}
 \label{Ruelle}
 In this section, we review some basic facts about the \emph{Ruelle zeta function}. Define for $s\in \mathbb C$ such that $\mathrm{Re}(s)\gg 1$, 
 \begin{equation}
\label{eq:Zeta}
\zeta_R^g(s):=\prod_{{\gamma^{\sharp}\in \Gamma^{\sharp}}}\big(1-e^{-s \ell_g({\gamma^{\sharp}})}\big), 
\end{equation}
here, $\Gamma^{\sharp}$ is the set of all \emph{primitive geodesics}, that is, geodesics with minimal periods. 
The Ruelle zeta function extends meromorphically to the complex plane \cite{DyaZw} and, if we denote by $m_R^g(\lambda_0)$ the multiplicity of $\lambda_0$ as a zero of $\zeta_R^g$  (with the convention that $m_R^g(\lambda_0)<0$ for a pole $\lambda_0$) then we have
\begin{equation}
\label{eq:zerozeta}
m_R^g(\lambda_0)=m_1^g(\lambda_0)-m_0^g(\lambda_0)-m_2^g(\lambda_0)=m_1^g(\lambda_0)-2m_0^g(\lambda_0),
\end{equation}
see \cite[Equation (3.1)]{zazi} and where we used that $m_0^g(\lambda_0)=m_2^g(\lambda_0)$.

\section{Proof of Theorem \ref{theo:hyp}}
In this section, we prove Theorem \ref{theo:hyp}. Let $(M^2,g_0)$ be a closed hyperbolic surface. Recall that for $s\in \mathbb C$, one has the relation
\begin{equation}
\label{eq:zetaRel}
\zeta_R^{g_0}(s)=\frac{Z_S^{g_0}(s)}{Z_S^{g_0}(s+1)}, 
\end{equation}
where $Z_S^{g_0}$ is the \emph{Selberg zeta function} of $g_0$, see for instance, \cite[p. 257]{Baladi98}. We will not recall the definition of the Selberg zeta function but recall that it admits a holomorphic extension to $\mathbb C$. Moreover, the zeros of $Z_S^{g_0}$ are located either on the real axis or on $\tfrac 12+i\mathbb R$. Let 
$$\mathcal R_0:=\{r\geq 0 \mid \tfrac 14+r^2\in \mathrm{Spec}(\Delta_{g_0})\}.$$ Then the $\tfrac 12\pm ir\in \tfrac 12+i\mathbb R$ for $r\in \mathcal R_0$ are zeros of $Z_S^{g_0}$ with multiplicity equal to 
$$\nu_0(r):=\mathrm{dim}\big(\mathrm{Ker}\big(\Delta_{g_0}-( \tfrac 14+r^2)\big)\big),$$ see \cite[Theorem 4.11]{Hej76}. In particular, we deduce from \eqref{eq:zetaRel} and the Weyl law (see for instance \cite[Theorem 3.5]{DG75}) that for any fixed $\delta>0$, we have 
\begin{equation}
\label{eq:Weyl}
-\sum_{\omega \leq r\leq \omega+\delta,\ r\in \mathcal R_0}m_R^{g_0}(-\tfrac 12+ir)= \sum_{\omega \leq r\leq \omega+\delta,\ r\in \mathcal R_0}\nu_0(r)=\frac{\mathrm{Area}_{g_0}(M)}{2\pi}\omega \delta+o(\omega),\quad \text{ when } \omega \to+\infty.
\end{equation}
The next lemma shows that if a negatively curved metric has the same length spectrum as $g_0$, then the number of Pollicott-Ruelle resonances for the action on functions on the axis $-\tfrac 12+i\mathbb R$ is bounded below by Weyl term.
\begin{lemma}
\label{lemm:Weyl}
Let $(M,g)$ be a negatively curved surface. Suppose that $L(g)=L(g_0)$. Then we have, for any $\delta>0$, 
$$
\sum_{\omega \leq r\leq \omega+\delta,\ r\in \mathcal R_0}m_0^g(-\tfrac 12+ir) \geq \frac{\mathrm{Area}_{g_0}(M)}{4\pi}\omega \delta+o(\omega),\quad \text{ when } \omega \to+\infty.
$$
\end{lemma}
\begin{proof}
Since $L(g)=L(g_0)$, \eqref{eq:Zeta} implies that $\zeta_R^g(s)=\zeta_R^{g_0}(s)$ when $\mathrm{Re}(s)\gg 1$. The meromorphic extensions of the two functions coincide by analytic continuation. In particular, using \eqref{eq:Weyl} and \eqref{eq:zerozeta}, we obtain
\begin{align*}
2&\sum_{\omega \leq r\leq \omega+\delta,\ r\in \mathcal R_0}m_0^{g}(-\tfrac 12+ir)\geq \sum_{\omega \leq r\leq \omega+\delta,\ r\in \mathcal R_0}\big(2m_0^{g}(-\tfrac 12+ir)-m_1^{g}(-\tfrac 12+ir)\big)
\\&=-\sum_{\omega \leq r\leq \omega+\delta,\ r\in \mathcal R_0}m_R^{g}(-\tfrac 12+ir)=-\sum_{\omega \leq r\leq \omega+\delta,\ r\in \mathcal R_0}m_R^{g_0}(-\tfrac 12+ir)=\frac{\mathrm{Area}_{g_0}(M)}{2\pi}\omega \delta+o(\omega).
\end{align*}
This concludes the proof of the lemma.
\end{proof}

In \cite{FT}, two sequences of numbers $(\gamma_k^\pm(g))_{k\in \mathbb N}$ associated with the geodesic flow of $g$ are defined. We will not recall the exact definitions and refer to \cite[Equation (1.6)]{FT}. We will need the following bounds on these numbers (see \cite[Equation $(1.7)$]{FT} and the remark after):
$$\forall k\geq 0, \quad -\big(\frac 12 +k)\lambda_{\mathrm{max}}(g)\leq \gamma_k^-(g)\leq \gamma_k^+(g)\leq - \big(\frac 12 +k)\lambda_{\mathrm{min}}(g),$$
where $\lambda_{\mathrm{min}}(g)$ (resp. $\lambda_{\mathrm{max}}(g)$) denote the minimal (resp. maximal) Lyapunov exponent of $(\varphi_t^g)_{t\in \mathbb R}$. Suppose that $g$ is $\alpha$-pinched, that is,
$$\exists C>0, \quad -C\leq K_g\leq -\alpha C. $$ Then there is $c>0$ such that for any $t\geq 0$ and any $p\in SM$,
\begin{equation}
\label{eq:kli}
{c}^{-1}e^{-\sqrt{C} t} \leq \| d\varphi_ t^g|_{E_s(p)}\|\leq ce^{-\sqrt{\alpha C} t}, \quad {c}^{-1}e^{-\sqrt C t} \leq \| d\varphi_ {-t}^g|_{E_u(p)}\|\leq ce^{-\sqrt{\alpha C} t},
\end{equation}
see for instance \cite[Theorem 3.9.1]{Kli}. In particular, we deduce that for any $\alpha$-pinched metric, we have
\begin{equation}
\label{eq:pinchingcond}
\forall k\geq 0, \quad -\big(\frac 12 +k)\sqrt{C}\leq \gamma_k^-(g)\leq \gamma_k^+(g)\leq  -\big(\frac 12 +k)\sqrt{\alpha C}.
\end{equation}
We now apply \cite[Theorem 1.13]{FT} to show that if $g$ is sufficiently pinched, the resonances in the \emph{first band} $[\gamma_0^-(g),\gamma_0^+(g)]\times i\mathbb R$ concentrate, when $\mathrm{Im}(z)\to+\infty$, near $-\tfrac {1}2h_{\mathrm{Liou}}(g)+i\mathbb R$.
\begin{proposition}
\label{prop:FT} Let $(M,g)$ be a strictly $\tfrac 19$-pinched metric on a closed surface. Then for any $\eps>0$ small enough, we have
$$\lim_{\delta\to+\infty}\limsup_{\omega \to +\infty}\frac{1}{ \omega \delta}\mathrm{Card}\big(\sigma(\mathbf X_0^g)\cap \mathrm{Strip}_{\omega,\delta}\big)=0, $$
where the resonances are counted with their algebraic multiplicity and where
$$ \mathrm{Strip}_{\omega,\delta}:=\big( [\gamma_0^-(g)-\epsilon,\gamma_0^+(g)+\epsilon]\setminus [-\tfrac {1}2h_{\mathrm{Liou}}(g)-\epsilon,-\tfrac {1}2 h_{\mathrm{Liou}}(g)+\epsilon]\big)\times i[\omega,\omega+\delta].$$
\end{proposition}
\begin{proof}
In order to apply \cite[Theorem 1.13]{FT} we need to verify that $g$ satisfies the \emph{pinching condition} $\gamma_1^+(g)\leq \gamma_0^-(g)$. Since $g$ is  strictly $\tfrac 19$-pinched, it is $\alpha$-pinched for some $\alpha>\tfrac 19$. We use \eqref{eq:pinchingcond} to obtain
\begin{equation}
\label{eq:upper}\gamma_1^+(g)\leq  -\frac 32 \sqrt{\alpha C}<-\frac 12 \sqrt{C}\leq \gamma_0^-(g). 
\end{equation}
By \cite[Theorem 1.13]{FT}, the resonances in the first band concentrate (in the sense made explicit in the statement of the proposition), when $\mathrm{Im}(z)\to +\infty$, near a narrower band $[\check{\gamma}_0^-(g),\check{\gamma}_0^+(g)]\times i\mathbb R$. Moreover, since the bundle is trivial, by the discussion after \cite[Theorem 1.13]{FT} we have 
$$\check{\gamma}_0^-(g)=\check{\gamma}_0^+(g)=\frac 12\int_{S^gM}J^s_g(v)d\mu_g(v)=-\frac 12 h_{\mathrm{Liou}}(g) ,$$
where we used \eqref{eq:SRBentropy}. This concludes the proof of the proposition.
\end{proof}
We can now conclude the proof of Theorem \ref{theo:hyp} combining the estimate obtained in Lemma \ref{lemm:Weyl} and the result of Proposition \ref{prop:FT}.
\begin{proof}[Proof of Theorem \ref{theo:hyp}]
Let $(M^2,g_0)$ be a closed hyperbolic surface. Let $g$ be a strictly $\tfrac 19$-pinched negatively curved metric on $M$ and suppose that $L(g)=L(g_0)$. From \eqref{eq:top2}, we first have
\begin{equation}
\label{eq:topequal}
h_{\mathrm{top}}(g)=h_{\mathrm{top}}(g_0)=1.
\end{equation} Recall that $g$ is $\alpha$-pinched for some $\alpha>\tfrac 19$, that is, 
$$\exists C>0, \quad -C\leq K_g\leq -\alpha C. $$
We claim that $C\geq 1$. Indeed, since $K_g\geq -C$, the Bishop-Gromov comparison theorem (see for instance \cite[Lemma 7.1.4]{Pet16}) gives $h_{\mathrm{vol}}(g)\leq \sqrt C$, where $h_{\mathrm{vol}}(g)$ is the volume entropy of $g$. In negative curvature, the volume entropy and topological entropy coincide \cite{Man79} and \eqref{eq:topequal} then implies that $C\geq 1$. 

We use \cite[Theorem 1.7]{FT} and Lemma \ref{lemm:Weyl} to deduce that $-\tfrac 12+i\mathbb R\subset [\gamma_k^-(g),\gamma_k^+(g)]\times i\mathbb R$ for some $k\geq 0$. Recall from \eqref{eq:upper} that $\gamma_1^+(g)<-\tfrac 12\sqrt C\leq -\tfrac 12$ which means that we must have
$$ -\tfrac 12+i\mathbb R\subset [\gamma_0^-(g),\gamma_0^+(g)]\times i\mathbb R.$$
Next, suppose for a contradiction that $h_{\mathrm{Liou}}(g)\neq 1$, then $-\frac 12 h_{\mathrm{Liou}}(g)\neq -\tfrac 12$ and we can choose $\epsilon>0$ small enough in Proposition \ref{prop:FT} such that for any $\omega,\delta>0$, we have 
$$ -\tfrac 12+i[\omega,\omega+\delta]\subset \mathrm{Strip}_{\omega,\delta}.$$
Applying Lemma \ref{lemm:Weyl} thus gives
\begin{align*} \lim_{\delta\to+\infty}\limsup_{\omega \to +\infty}\frac{1}{ \omega \delta}\mathrm{Card}\big(\sigma(\mathbf X_0^g)\cap \mathrm{Strip}_{\omega,\delta}\big)&\geq \lim_{\delta\to+\infty}\limsup_{\omega \to +\infty}\frac{1}{ \omega \delta}\sum_{\omega \leq r\leq \omega+\delta,\ r\in \mathcal R_0}m_0^{g}(-\tfrac 12+ir) 
\\&\geq \frac{\mathrm{Area}_{g_0}(M)}{4\pi}>0.
\end{align*}
This contradicts the statement of Proposition \ref{prop:FT} and thus $h_{\mathrm{Liou}}(g)= 1$. We have shown that 
$h_{\mathrm{Liou}}(g)= 1= h_{\mathrm{top}}(g),$
which implies that $g$ is hyperbolic by \cite{Ka} and this concludes the proof of the theorem.
\end{proof}
The following lemma is the last ingredient needed to show Theorem \ref{theo:main}. \begin{lemma}
\label{lemma:final}
Let $(M^2,g_0)$ be a closed hyperbolic surface. There is $\epsilon>0$ such that if $g_1$ is another hyperbolic metric with $\|g_1-g_0\|_{C^0}<\epsilon$ and $L(g_1)=L(g_0)$, then $g_0$ and $g_1$ are isometric.
\end{lemma}
\begin{proof}
Suppose that $L(g_1)=L(g_0)$ and that $\|g_1-g_0\|_{C^0}<\epsilon$, then
$$\mathcal L_{g_1}(c)\leq \int_0^{\ell_{g_0}(\gamma_{g_0}(c))}\sqrt{g_1( \dot\gamma_{g_0}(c),\dot\gamma_{g_0}(c))}dt\leq \int_0^{\ell_{g_0}(\gamma_{g_0}(c))}\sqrt{1+\epsilon}dt=\sqrt{1+\epsilon}\mathcal L_{g_0}(c).$$
A similar computation gives a lower bound and we have 
$$\forall c\in \mathcal C, \quad \sqrt{1-\epsilon}\mathcal L_{g_0}(c)\leq \mathcal L_{g_1}(c) \leq \sqrt{1+\epsilon}\mathcal L_{g_0}(c). $$
Since the length spectrum is discrete, we deduce that
$$\forall c\in \mathcal C, \ \exists \epsilon_c>0, \ L(g_0)=L(g_1), \|g_0-g_1\|_{C^0}<\epsilon_c \ \Rightarrow \  \mathcal L_{g_0}(c)= \mathcal L_{g_1}(c).$$
Now, for hyperbolic metrics, there is a finite set $F\subset \mathcal C$ such that the marked length spectrum of the metric is determined from its restriction to $F$, see for instance \cite[Theorem 10.7]{FaMa}. In particular, if $\epsilon>0$ is chosen so that $0<\epsilon<\min_{c\in F}\epsilon_c$, then $L(g_1)=L(g_0)$ implies $\mathcal L_{g_0}=\mathcal L_{g_1}$. This implies that $g_0$ and $g_1$ are isometric by \cite{Cro90,Ota90}.
\end{proof}

\bibliography{biblio}

@article{TsuZh,
author = {Tsujii, Masato and Zhang, Zhiyuan},
copyright = {Distributed under a Creative Commons Attribution 4.0 International License},
issn = {0003-486X},
journal = {Annals of Mathematics},
language = {eng},
number = {1},
publisher = {Princeton University, Department of Mathematics},
title = {Smooth mixing {A}nosov flows in dimension three are exponentially mixing},
volume = {197},
year = {2023},
}

@article{zazi,
  title={Ruelle zeta function at zero for surfaces},
  author={Semyon Dyatlov and Maciej Zworski},
  journal={Inventiones mathematicae},
  year={2016},
  volume={210},
  pages={211-229},
  url={https://api.semanticscholar.org/CorpusID:23047596}
}

@article{Reid1992,
  author  = {Reid, Alan W.},
  title   = {Isospectrality and commensurability of arithmetic hyperbolic
             {$2$}- and {$3$}-manifolds},
  journal = {Duke Mathematical Journal},
  volume  = {65},
  number  = {2},
  pages   = {215--228},
  year    = {1992},
  doi     = {10.1215/S0012-7094-92-06508-2}
}

@article{Reid2014,
  author  = {Reid, Alan W.},
  title   = {Traces, lengths, axes and commensurability},
  journal = {Annales de la Facult{\'e} des Sciences de Toulouse.
             Math{\'e}matiques},
  series  = {6},
  volume  = {23},
  number  = {5},
  pages   = {1103--1118},
  year    = {2014},
  doi     = {10.5802/afst.1438}
}

@article{Butt26,
  author  = {Butt, Karen},
  title   = {Approximate control of the marked length spectrum by short geodesics},
  journal = {Ergodic Theory and Dynamical Systems},
  volume  = {46},
  number  = {2},
  pages   = {490--513},
  year    = {2026},
  doi     = {10.1017/etds.2025.10212}
}

@article{Schenck2020,
  author  = {Schenck, Emmanuel},
  title   = {Exponential gaps in the length spectrum},
  journal = {Journal of Modern Dynamics},
  volume  = {16},
  year    = {2020},
  pages   = {207--223},
  doi     = {10.3934/jmd.2020007}
}

@article{Man79,
  author  = {Manning, Anthony},
  title   = {Topological entropy for geodesic flows},
  journal = {Ann. of Math. (2)},
  volume  = {110},
  number  = {3},
  pages   = {567--573},
  year    = {1979},
  doi     = {10.2307/1971239}
}

@book{Pet16,
  author    = {Petersen, Peter},
  title     = {Riemannian Geometry},
  series    = {Graduate Texts in Mathematics},
  volume    = {171},
  edition   = {Third},
  publisher = {Springer},
  address   = {Cham},
  year      = {2016},
  doi       = {10.1007/978-3-319-26654-1}
}

@book{Kli,
  author    = {Klingenberg, Wilhelm},
  title     = {Riemannian Geometry},
  series    = {De Gruyter Studies in Mathematics},
  volume    = {1},
  publisher = {Walter de Gruyter},
  address   = {Berlin and New York},
  year      = {1982}
}

@book{Hej76,
  author    = {Hejhal, Dennis A.},
  title     = {The Selberg Trace Formula for
               {$\mathrm{PSL}(2,\mathbb{R})$}. Volume I},
  series    = {Lecture Notes in Mathematics},
  volume    = {548},
  publisher = {Springer-Verlag},
  address   = {Berlin},
  year      = {1976},
  doi       = {10.1007/BFb0079608},
  isbn      = {978-3-540-07988-0}
}

@article{Baladi98,
  author  = {Baladi, Viviane},
  title   = {Periodic Orbits and Dynamical Spectra},
  journal = {Ergodic Theory and Dynamical Systems},
  volume  = {18},
  number  = {2},
  pages   = {255--292},
  year    = {1998},
  doi     = {10.1017/S0143385798113925}
}

@article{DyaZw,
author = {Dyatlov, Semyon and Zworski, Maciej},
issn = {0012-9593},
journal = {Annales scientifiques de l'École normale supérieure},
language = {eng},
number = {3},
pages = {543-577},
publisher = {Société Mathématique de France},
title = {Dynamical zeta functions for Anosov flows via microlocal analysis},
volume = {49},
year = {2016},
}

@article{Liv,
author = {Liverani, Carangelo},
address = {Princeton, NJ},
copyright = {Copyright 2004 Princeton University (Mathematics Department)},
issn = {0003-486X},
journal = {Annals of Mathematics},
number = {3},
pages = {1275-1312},
publisher = {Princeton University Press},
title = {On Contact {A}nosov Flows},
volume = {159},
year={2004}
}

@article{GL19,
  title={The marked length spectrum of {A}nosov manifolds},
  author={Guillarmou, Colin and Lefeuvre, Thibault},
  journal={Annals of Mathematics},
  volume={190},
  number={1},
  pages={321--344},
  year={2019},
  publisher={JSTOR}
}

@book{FaMa,
  author    = {Farb, Benson and Margalit, Dan},
  title     = {A primer on mapping class groups},
  series    = {Princeton Mathematical Series},
  volume    = {49},
  publisher = {Princeton University Press},
  address   = {Princeton, NJ},
  year      = {2012},
  isbn      = {978-0-691-14794-9}
}

@article{DG75,
  author  = {Duistermaat, J. J. and Guillemin, V. W.},
  title   = {The spectrum of positive elliptic operators and periodic bicharacteristics},
  journal = {Invent. Math.},
  volume  = {29},
  number  = {1},
  pages   = {39--79},
  year    = {1975},
  doi     = {10.1007/BF01405172}
}

@article{Hu,
  author  = {Huber, Heinz},
  title   = {Zur analytischen Theorie hyperbolischer {Raumformen} und {Bewegungsgruppen}},
  journal = {Math. Ann.},
  volume  = {138},
  pages   = {1--26},
  year    = {1959},
  doi     = {10.1007/BF01369663}
}

@article{Sha09,
  author  = {Sharafutdinov, Vladimir A.},
  title   = {Local audibility of a hyperbolic metric},
  journal = {Siberian Math. J.},
  volume  = {50},
  number  = {5},
  pages   = {929--944},
  year    = {2009},
  doi     = {10.1007/s11202-009-0103-7}
}

@article{FT,
  author  = {Faure, Fr{\'e}d{\'e}ric and Tsujii, Masato},
  title   = {Micro-local analysis of contact {Anosov} flows and band
             structure of the {Ruelle} spectrum},
  journal = {Communications of the American Mathematical Society},
  volume  = {4},
  number  = {15},
  pages   = {641--745},
  year    = {2024},
  doi     = {10.1090/cams/40}
}

@misc{Gogolev18,
  author       = {Gogolev, Andrey},
  title        = {Length Spectrum Rigidity},
  year         = {2018},
  month        = aug,
  day          = {21},
  howpublished = {Blog post},
  url          = {https://andreygogolev.wordpress.com/2018/08/21/length-spectrum-rigidity/},
  urldate      = {2026-09-22}
}

@misc{DWWMRO,
      title={A finite Livshits theorem and local length spectrum rigidity}, 
      author={Jonathan DeWitt and Spencer Durham and James Marshall Reber and Thomas Aloysius O'Hare},
      year={2026},
      eprint={2609.23762},
      archivePrefix={arXiv},
      primaryClass={math.DS},
      url={https://arxiv.org/abs/2609.23762}, 
}

@article{DJ16,
title = {On small gaps in the length spectrum},
journal = {Journal of Modern Dynamics},
volume = {10},
number = {2},
pages = {339-352},
year = {2016},
issn = {1930-5311},
doi = {10.3934/jmd.2016.10.339},
url = {https://www.aimsciences.org/article/id/0d52ef5f-746f-4f41-b791-e02e06d8aad2},
author = {Dmitry  Dolgopyat and Dmitry Jakobson}
}

@article{GuKnLef, title={Geodesic stretch, pressure metric and marked length spectrum rigidity}, volume={42}, DOI={10.1017/etds.2021.75}, number={3}, journal={Ergodic Theory and Dynamical Systems}, author={Guillarmou, Colin and Knieper, Gerhard and Lefeuvre, Thibault}, year={2022}, pages={974–1022}}

@articleInfo{BL,
title = {Smooth Anosov flows: Correlation spectra and stability},
journal = {Journal of Modern Dynamics},
volume = {1},
number = {2},
pages = {301-322},
year = {2007},
issn = {1930-5311},
doi = {10.3934/jmd.2007.1.301},
url = {https://www.aimsciences.org/article/id/b30e2a63-0216-4f6a-bd53-590a073d927d},
author = {Oliver Butterley and Carlangelo Liverani}
}

@article{BKL,
doi = {10.1088/0951-7715/15/6/309},
url = {https://dx.doi.org/10.1088/0951-7715/15/6/309},
year = {2002},
month = {sep},
publisher = {},
volume = {15},
number = {6},
pages = {1905},
author = {Michael Blank and Gerhard Keller and Carlangelo Liverani},
title = {Ruelle–Perron–Frobenius spectrum for Anosov maps},
journal = {Nonlinearity}
}

@article{BT,
     author = {Baladi, Viviane and Tsujii, Masato},
     title = {Anisotropic {H\"older} and {Sobolev} spaces for hyperbolic diffeomorphisms},
     journal = {Annales de l'Institut Fourier},
     pages = {127--154},
     publisher = {Association des Annales de l{\textquoteright}institut Fourier},
     volume = {57},
     number = {1},
     year = {2007},
     doi = {10.5802/aif.2253},
     zbl = {1138.37011},
     mrnumber = {2313087},
     language = {en},
     url = {http://www.numdam.org/articles/10.5802/aif.2253/}
}

@article{burnskatok,
  title={Manifolds with non-positive curvature},
  author={Burns, Keith and Katok, Anatole and Ballman, Werner and Brin, Michael and Eberlein, Patrick and Osserman, Robert},
  journal={Ergodic Theory and Dynamical Systems},
  volume={5},
  number={2},
  pages={307--317},
  year={1985},
  publisher={Cambridge University Press}
}

@article{GuLefPa,
  title={Marked length spectrum rigidity for {A}nosov surfaces},
  author={Guillarmou, Colin and Lefeuvre, Thibault and Paternain, Gabriel P},
  journal={Duke Mathematical Journal},
  volume={174},
  number={1},
  pages={131--157},
  year={2025},
  publisher={Duke University Press}
}

@article{Vig80,
  author  = {Vign{\'e}ras, Marie-France},
  title   = {Vari{\'e}t{\'e}s riemanniennes isospectrales et non isom{\'e}triques},
  journal = {Ann. of Math. (2)},
  volume  = {112},
  number  = {1},
  pages   = {21--32},
  year    = {1980},
  doi     = {10.2307/1971319}
}

@article{Ka,
title = "Entropy and closed geodesics",
author = "Anatole Katok",
year = "1982",
volume = "2",
pages = "339--365",
journal = "Ergodic Theory and Dynamical Systems",
issn = "0143-3857",
publisher = "Cambridge University Press",
number = "3-4",
}

@article{GK80,
  title={Some inverse spectral results for negatively curved 2-manifolds},
  author={Guillemin, Victor and Kazhdan, David},
  journal={Topology},
  volume={19},
  number={3},
  pages={301--312},
  year={1980},
  publisher={Pergamon}
}

@article{CS98,
  author  = {Croke, Christopher B. and Sharafutdinov, Vladimir A.},
  title   = {Spectral rigidity of a compact negatively curved manifold},
  journal = {Topology},
  volume  = {37},
  number  = {6},
  pages   = {1265--1273},
  year    = {1998},
  doi     = {10.1016/S0040-9383(97)00086-4}
}

@article{ota90,
  title={Le spectre marqu{\'e} des longueurs des surfaces {\`a} courbure n{\'e}gative},
  author={Otal, Jean-Pierre},
  journal={Annals of Mathematics},
  volume={131},
  number={1},
  pages={151--162},
  year={1990},
  publisher={JSTOR}
}

@article{cro90,
  title={Rigidity for surfaces of non-positive curvature},
  author={Croke, Christopher B},
  journal={Commentarii Mathematici Helvetici},
  volume={65},
  number={1},
  pages={150--169},
  year={1990},
  publisher={Springer}
}

@article{ham99,
  title={Cocycles, symplectic structures and intersection},
  author={Hamenst{\"a}dt, Ursula},
  journal={Geometric \& Functional Analysis GAFA},
  volume={9},
  number={1},
  pages={90--140},
  year={1999},
  publisher={Springer}
}

@article{But25,
  author  = {Butt, Karen},
  title   = {Quantitative marked length spectrum rigidity},
  journal = {Geom. Topol.},
  volume  = {29},
  number  = {8},
  pages   = {3995--4054},
  year    = {2025},
  doi     = {10.2140/gt.2025.29.3995}
}

@misc{but1,
      title={Quantitative marked length spectrum rigidity for surfaces}, 
      author={Karen Butt},
      year={2025},
      eprint={2509.16829},
      archivePrefix={arXiv},
      primaryClass={math.DG},
      url={https://arxiv.org/abs/2509.16829}, 
}

@article{Sun85,
  author  = {Sunada, Toshikazu},
  title   = {Riemannian coverings and isospectral manifolds},
  journal = {Ann. of Math. (2)},
  volume  = {121},
  number  = {1},
  pages   = {169--186},
  year    = {1985},
  doi     = {10.2307/1971195}
}

@article{GouLiv,
author = {Gouëzel, Sébastien and Liverani, Carlangelo},
copyright = {Copyright 2008 Lehigh University},
issn = {0022-040X},
journal = {Journal of differential geometry},
language = {eng},
number = {3},
pages = {433-477},
publisher = {Lehigh University},
title = {Compact locally maximal hyperbolic sets for smooth maps: fine statistical properties},
volume = {79},
year = {2008},
}

@misc{Fau08,
      title={Semi-classical approach for Anosov diffeomorphisms and Ruelle resonances}, 
      author={Faure, Frederic and Roy, Nicolas and Sjöstrand,Johannes},
      year={2008},
      eprint={0802.1780},
      archivePrefix={arXiv},
      primaryClass={nlin.CD}
}

@article{Fau10,
author = {Faure, Frédéric and Sjöstrand, Johannes},
address = {Berlin/Heidelberg},
copyright = {Springer-Verlag 2011},
issn = {0010-3616},
journal = {Communications in Mathematical Physics},
number = {2},
pages = {325-364},
publisher = {Springer-Verlag},
title = {Upper Bound on the Density of {R}uelle Resonances for {A}nosov Flows},
volume = {308},
year = {2011},
}

@book{GuMaz,
    author ={Guillarmou, Colin and Mazzuchelli, Marco} ,
    title = {An introduction to geometric inverse problems},
     pubsliher={to appear in the series Graduate Studies in Mathematics, AMS},
     year={2026}
}

@book{FishHas,
author = {Fisher, Todd and Hasselblatt, Boris},
address = {Zuerich, Switzerland},
isbn = {3-03719-700-5},
language = {eng},
publisher = {European Mathematical Society Publishing House},
series = {Zurich Lectures in Advanced Mathematics (ZLAM)},
title = {Hyperbolic Flows},
year = {2019 - 1210},
}
\bibliographystyle{alpha}

\end{document}